\documentclass{amsart}
\usepackage{amssymb, amsmath, mathrsfs, verbatim, url, bbm, bbold}

\makeatletter
\@namedef{subjclassname@2020}{%
  \textup{2020} Mathematics Subject Classification}
\makeatother

\theoremstyle{plain}
\newtheorem{theorem}{Theorem}[section]
\newtheorem{lemma}[theorem]{Lemma}
\newtheorem{proposition}[theorem]{Proposition}
\newtheorem{corollary}[theorem]{Corollary}

\theoremstyle{definition}
\newtheorem{definition}[theorem]{Definition}

\newcommand{\re}{\upharpoonright}
\newcommand{\dom}{\mathsf{dom}}
\newcommand{\ran}{\mathsf{ran}}
\newcommand{\maxi}{\mathsf{max}}
\newcommand{\di}{\mathsf{d}}
\newcommand{\Ball}{\mathsf{B}}
\newcommand{\cl}{\mathsf{cl}}
\newcommand{\dime}{\mathsf{dim}}
\newcommand{\coord}{\mathsf{coord}}
\newcommand{\supp}{\mathsf{supp}}
\newcommand{\CDH}{\mathsf{CDH}}
\newcommand{\GPP}{\mathsf{GPP}}
\newcommand{\WGPP}{\mathsf{WGPP}}

\newcommand{\BB}{\mathcal{B}}
\newcommand{\CC}{\mathcal{C}}
\newcommand{\DD}{\mathcal{D}}
\newcommand{\FF}{\mathcal{F}}
\newcommand{\HH}{\mathcal{H}}
\newcommand{\Ss}{\mathcal{S}}
\newcommand{\TT}{\mathcal{T}}
\newcommand{\BBB}{\mathbb{B}}
\newcommand{\DDD}{\mathbb{D}}
\newcommand{\PPP}{\mathbb{P}}
\newcommand{\QQQ}{\mathbb{Q}}
\newcommand{\RRR}{\mathbb{R}}
\newcommand{\SSS}{\mathbb{S}}
\newcommand{\cccc}{\mathfrak{c}}
\newcommand{\pppp}{\mathfrak{p}}

\begin{document}

\title{Countable dense homogeneity in large products of Polish spaces}

\author{Andrea Medini}
\address{Institut f\"{u}r Diskrete Mathematik und Geometrie
\newline\indent Technische Universit\"{a}t Wien
\newline\indent Wiedner Hauptstra\ss e 8-10/104
\newline\indent 1040 Vienna, Austria}
\email{andrea.medini@tuwien.ac.at}
\urladdr{https://www.dmg.tuwien.ac.at/medini/}

\author{Juris Stepr\={a}ns}
\address{Department of Mathematics
\newline\indent York University
\newline\indent 4700, Keele Street
\newline\indent Toronto, Ontario M3J 1P3, Canada}
\email{steprans@yorku.ca}
\urladdr{https://www.yorku.ca/professor/steprans/}

\date{October 28, 2025}

\thanks{The first-listed author thanks Vadim Kulikov for a useful conversation in their old office at the Kurt G\"{o}del Research Center. This research was funded in whole or in part by the Austrian Science Fund (FWF) DOI 10.55776/P35588. For open access purposes, the authors have applied a CC BY public copyright license to any authors-accepted manuscript version arising from this submission.}

\begin{abstract}
We give a unified treatment of the countable dense homogeneity of products of Polish spaces, with a focus on uncountable products. Our main result states that a product of fewer than $\pppp$ Polish spaces is countable dense homogeneous if the following conditions hold:
\begin{itemize}
\item Each factor is strongly locally homogeneous,
\item Each factor is strongly $n$-homogeneous for every $n\in\omega$,
\item Every countable subset of the product can be brought in general position.
\end{itemize}
For example, using the above theorem, one can show that $2^\kappa$, $\omega^\kappa$, $\RRR^\kappa$ and $[0,1]^\kappa$ are countable dense homogeneous for every infinite $\kappa <\pppp$ (these results are due to Stepr\={a}ns and Zhou, except for the one concerning $\omega^\kappa$). In fact, as a new application, we will show that every product of fewer than $\pppp$ connected manifolds with boundary is countable dense homogeneous, provided that none or infinitely many of the boundaries are non-empty. This generalizes a result of Yang. Along the way, we will discuss and employ several results concerning the general position of countable sets. Finally, we will show that our main result and its corollaries are optimal.
\end{abstract}

\subjclass[2020]{54B10, 54E50, 54H11, 57N99, 03E17.}

\keywords{Countable dense homogeneous, infinite product, Polish space, general position, topological group, manifold with boundary, zero-dimensional.}

\maketitle

\tableofcontents

\section{Introduction}\label{section_introduction}

By \emph{space} we mean topological space, with no assumptions regarding separation axioms. By \emph{countable} we mean at most countable. Given a space $X$, we will denote by $\HH(X)$ the group of all homeomorphisms $h:X\longrightarrow X$ with the operation of composition. A separable space $X$ is \emph{countable dense homogeneous} ($\CDH$) if for every pair $(D,E)$ of countable dense subsets of $X$ there exists $h\in\HH(X)$ such that $h[D]=E$. See \cite[Sections 14-16]{arhangelskii_van_mill} and \cite{hrusak_van_mill} respectively for a survey and a collection of open problems on this topic.

The following \cite[Theorem 5.2]{anderson_curtis_van_mill} is the fundamental positive result in the theory of $\CDH$ spaces. Given a space $X$ and $h\in\HH(X)$, we will use the notation
$$
\supp(h)=\{x\in X:h(x)\neq x\}
$$
for the \emph{support} of $h$. A space $X$ is \emph{strongly locally homogeneous} if there exists a base $\BB$ for $X$ such that for every $U\in\BB$ and $(x,y)\in U\times U$ there exists $h\in\HH(X)$ such that $h(x)=y$ and $\supp(h)\subseteq U$.

\begin{theorem}[Anderson, Curtis, van Mill]\label{theorem_fundamental}
Every strongly locally homogeneous Polish space is $\CDH$.
\end{theorem}

In particular, familiar spaces like manifolds, the Cantor set $2^\omega$ and the Baire space $\omega^\omega$ are all examples of $\CDH$ spaces (see Proposition \ref{proposition_zero_dimensional}). Another example is given by the Hilbert cube $[0,1]^\omega$, as shown by Fort \cite[Theorem 2]{fort}. In fact, the following more general result \cite[Corollary 2 and Remark 2]{yang} holds.\footnote{\,In \cite[Remark 2]{yang}, the case in which all boundaries are empty is not mentioned; but this is in fact the easiest case, since the tricky part is dealing with points on the boundary.} Given a cardinal $\kappa$ and manifolds with boundary $X_\alpha$ for $\alpha\in\kappa$, it will be convenient to define the following condition, where $\partial X_\alpha$ denotes the boundary of $X_\alpha$:

\begin{enumerate}
\item[$(\partial)$]	Either $\partial X_\alpha\neq\varnothing$ for infinitely many values of $\alpha$ or $\partial X_\alpha=\varnothing$ for all $\alpha$.
\end{enumerate}

\begin{theorem}[Yang]\label{theorem_yang}
Let $\kappa\leq\omega$ be a cardinal, and let $X_\alpha$ for $\alpha\in\kappa$ be manifolds with boundary. Assume that each $X_\alpha$ is connected and that condition $(\partial)$ holds. Then $\prod_{\alpha\in\kappa}X_\alpha$ is $\CDH$.
\end{theorem}

As observed by Yang \cite[Example 1]{yang}, the assumption of connectedness cannot be dropped (see also Proposition \ref{proposition_optimality_connectedness}). Furthermore, in \cite[Remark 2]{yang}, a result of Ganea \cite{ganea} is cited, from which one can easily deduce that the assumption of condition $(\partial)$ cannot be dropped either (see Theorem \ref{theorem_optimality_boundaries}).

Given that the product of at most $\cccc$ separable spaces is separable by the classical Hewitt-Marczewski-Pondiczery theorem (see \cite[Corollary 2.3.16]{engelking} or \cite[Exercise III.2.13]{kunen}), it is natural to wonder whether the $\kappa$-th power of the spaces mentioned above might be $\CDH$ for uncountable values of $\kappa$. For some of these spaces, this was consistently achieved by Stepr\={a}ns and Zhou \cite[Theorem 4]{steprans_zhou}, who showed that $2^\kappa$, $\RRR^\kappa$ and $[0,1]^\kappa$ are $\CDH$ for every infinite $\kappa<\pppp$ (see also \cite[Section 6]{hernandez_gutierrez_function} for the case of $[0,1]^\kappa$). As usual, we denote by $\pppp$ the \emph{pseudointersection number} (see \cite{kunen} for its definition and basic properties).

Although the article \cite{steprans_zhou} contains the fundamental idea, it does not provide a unique result from which the countable dense homogeneity of all of these spaces can be simultaneously deduced. Furthermore, details are not given on how to bring countable sets in general position, which is the crucial intermediate step in the proof.

The purpose of this article is to formulate such a unified approach (see Section~\ref{section_main}), and to systematically discuss the topic of general position (see Sections \ref{section_general}, \ref{section_groups} and~\ref{section_manifolds}). While working on this goal, we conjectured that it should be possible to extend the results of \cite{steprans_zhou} to the powers of other nice spaces, such as $\omega^\kappa$ or $X^\kappa$ for a connected manifold $X$. In fact, as applications of our main result, we will show that ``almost'' every infinite product of fewer than $\pppp$ zero-dimensional Polish spaces is $\CDH$ (see Corollary \ref{corollary_main_zero_dimensional}), and that Theorem \ref{theorem_yang} can be generalized to products with fewer than $\pppp$ factors (see Corollary \ref{corollary_main_manifolds}). Finally, in  Section \ref{section_optimality}, we will provide counterexamples to show that these results are optimal.

\section{More preliminaries, notation and terminology}\label{section_preliminaries}

Our reference for set theory is \cite{kunen}. Given a set $X$ and a cardinal $\kappa$, we will use the notation
$$
[X]^\kappa=\{S\subseteq X:|S|=\kappa\}\text{ and }[X]^{<\kappa}=\{S\subseteq X:|S|<\kappa\}.
$$
Given a cardinal $\kappa$ and a product $X=\prod_{\alpha\in\kappa}X_\alpha$, we will denote by $\pi_\alpha:X\longrightarrow X_\alpha$ the natural projection. Furthermore, given $\Omega\subseteq\kappa$, we will use the notation $X_\Omega=\prod_{\alpha\in\Omega}X_\alpha$ for the corresponding subproduct, and the notation $\tau_\Omega:X\longrightarrow X_\Omega$ for the natural projection.

Our references for general topology are \cite{engelking} and \cite{van_mill_book}. Our reference for topological groups is \cite{arhangelskii_tkachenko}. A space $X$ is \emph{crowded} if $X$ is non-empty and every non-empty open subset of $X$ is infinite.\footnote{\,We are slightly ``cheating'' with this definition, since one usually asks that $X$ is non-empty and has no isolated points. The reason for this deviation is that it makes the proof of Lemma \ref{lemma_exists_gp} work. In any case, the two definitions are equivalent for $\mathsf{T}_1$ spaces.} A subset $M$ of a space $Z$ is \emph{meager} if there exist closed nowhere dense subsets $C_n$ of $X$ for $n\in\omega$ such that $M\subseteq\bigcup_{n\in\omega}C_n$. A space $X$ is \emph{Baire} if every non-empty open subset of $X$ is non-meager in $X$. Given $n\in\omega$, a space $X$ is \emph{strongly $n$-homogeneous} if for every bijection $\sigma:F\longrightarrow G$, where $F,G\in [X]^n$, there exists $h\in\HH(X)$ such that $\sigma\subseteq h$. Trivially, every space is strongly $0$-homogeneous. A space is \emph{homogeneous} if it is strongly $1$-homogeneous. Every topological group is an example of homogeneous space.

A space is \emph{zero-dimensional} if it non-empty, $\mathsf{T}_1$, and it has a base consisting of clopen sets. It is easy to realize that every zero-dimensional space is Tychonoff. As the following well-known result shows, homogeneity has very strong consequences in the zero-dimensional realm (see \cite[proof of Lemma 1.9.4]{van_mill_book} for the first part; the second part can be proved using similar ideas).

\begin{proposition}\label{proposition_zero_dimensional}
Let $X$ be a zero-dimensional space. If $X$ is homogeneous then the following conditions hold:
\begin{itemize}
\item $X$ is strongly locally homogeneous,
\item $X$ is strongly $n$-homogeneous for every $n\in\omega$.	
\end{itemize}
\end{proposition}

Another useful result concerning homogeneity is the following, whose simple proof is left to the reader. We will use $\oplus$ to denote the topological disjoint sum.

\begin{proposition}\label{proposition_slh}
Let $X$ be a space, and let $Z$ denote the set of isolated points of~$X$. If $X$ is strongly locally homogeneous then $X$ is homeomorphic to $Z\oplus (X\setminus Z)$.
\end{proposition}

We will assume familiarity with the basic theory of Polish spaces (all the relevant facts are contained in \cite[Appendix A.6]{van_mill_book}).\footnote{\,In \cite{van_mill_book}, the terminology ``topologically complete'' is used instead of ``Polish'' (at the beginning of the book, van Mill assumes that all spaces are separable and metrizable).} A metric on a space $X$ is \emph{admissible} if it induces the topology of $X$. A space $X$ is \emph{Polish} if $X$ is separable and there exists an admissible complete metric on $X$.

We will assume some familiarity with the basic notions concerning topological manifolds, for which we refer to \cite{lee}. In particular, all manifolds with boundary are separable and metrizable (in fact, they are Polish). Furthermore, every non-empty manifold with boundary $X$ is an $n$-manifold with boundary for a unique $n\in\omega$, which we will denote by $\dime(X)$. Notice that a manifold with boundary $X$ is crowded if and only if it is non-empty and $\dime(X)\geq 1$. Unlike \cite{fort} and \cite{yang}, we allow manifolds with boundary to have an empty boundary, so that every manifold is a manifold with boundary (this includes the statement of Theorem \ref{theorem_yang}).

Given a manifold with boundary $X$, we will denote by $\partial X$ the boundary of $X$. Observe that
$$
\partial (X\times Y)=\big(\partial (X)\times Y\big)\cup\big(X\times\partial (Y)\big)
$$
whenever $X$ and $Y$ are manifolds with boundary, and that a similar formula holds for arbitrary finite products. With a slight abuse of notation, given an infinite cardinal $\kappa$ and manifolds with boundary $X_\alpha$ for $\alpha\in\kappa$, we will still use $\partial X$ to denote the \emph{pseudoboundary} of $X=\prod_{\alpha\in\kappa}X_\alpha$, that is
$$
\partial X=\{x\in X:x(\alpha)\in\partial X_\alpha\text{ for some }\alpha\in\kappa\}.
$$

Given $1\leq n<\omega$, we will use the following notation:
\begin{itemize}
\item $\BBB^n=\{x\in\RRR^n:||x||<1\}$,
\item $\DDD^n=\{x\in\RRR^n:||x||\leq 1\}$,
\item $\SSS^{n-1}=\{x\in\RRR^n:||x||=1\}$,
\end{itemize}
where $||x||=\sqrt{x(0)^2+\cdots +x(n-1)^2}$ denotes the Euclidean norm of $x\in\RRR^n$.

We will assume familiarity with the basic theory of pointwise and uniform convergence, for which we refer to \cite[Section 1.3]{van_mill_book} and \cite[Chapter 7]{munkres}. Given sets $X$ and $Y$, we will denote by $Y^X$ the collection of all functions $f:X\longrightarrow Y$. Given spaces $X$ and $Y$, we will use the notation
$$
\CC(X,Y)=\{f\in Y^X:f\text{ is continuous}\}.
$$

In Section \ref{section_convergence}, we will need the following fact, which can be safely assumed to be folklore. We leave its standard proof to the reader.\footnote{\,In the statement of Proposition \ref{proposition_convergence}, the sequence $(2^{-n}:n\in\omega)$ could have been substituted by any sequence $(\varepsilon_n:n\in\omega)$ such that each $\varepsilon_n >0$ and $\varepsilon_0 +\varepsilon_1 +\cdots$ converges. The same holds for Theorem \ref{theorem_convergence}.}

\begin{proposition}\label{proposition_convergence}
Let $X$ be a space, let $Y$ be a complete metric space with distance~$\di$, and let $f_n\in\CC(X,Y)$ for $n\in\omega$. Assume that $\di\big(f_{n+1}(x),f_n(x)\big)\leq 2^{-n}$ for every $n\in\omega$ and $x\in X$. Then $(f_n:n\in\omega)$ converges uniformly (to a continuous function).
\end{proposition}

In Section \ref{section_manifolds}, we will need the following well-known facts relating uniform convergence and equicontinuity. Their proofs are exercises in the ``$\varepsilon/3$ strategy,'' and we leave them to the reader. Given a space $X$, a point $x_0\in X$, and a metric space $Y$ with distance $\di$, recall that a family $\FF\subseteq Y^X$ is \emph{equicontinuous at $x_0$} if for every $\varepsilon>0$ there exists a neighborhood $U$ of $x_0$ in $X$ such that
$$
\di\big(f(x),f(x_0)\big)<\varepsilon\text{  whevener }f\in\FF\text{ and }x\in U.
$$
Such a family is \emph{equicontinuous} if it is equicontinuous at every point of $X$.\footnote{\,Our definition (which is taken from \cite[page 276]{munkres}) is sometimes expressed by saying that $\FF$ is \emph{pointwise equicontinuous}. Other sources give \emph{uniform equicontinuity} as their official definition instead (see for example \cite[Definition 7.22]{rudin}). In any case, it is not hard to see that these two notions are equivalent when $X$ is compact.}

\begin{proposition}\label{proposition_equicontinuous}
Let $X$ be a space, let $Y$ be a metric space, and let $f_n\in\CC(X,Y)$ for $n\in\omega$. If $(f_n:n\in\omega)$ converges uniformly (to a continuous function) then $\{f_n:n\in\omega\}$ is equicontinuous.
\end{proposition}

\begin{theorem}\label{theorem_dini}
Let $X$ be a compact space, let $Y$ be a metric space, and let $f_n\in\CC(X,Y)$ for $n\in\omega$. Assume that the following conditions hold:
\begin{itemize}
\item $(f_n:n\in\omega)$ converges to a continuous function $f$,
\item $\{f_n:n\in\omega\}$ is equicontinuous.
\end{itemize}
Then $(f_n:n\in\omega)$ converges uniformly (to $f$).
\end{theorem}

Let $X$ be a compact space, and let $Y$ be a metric space with distance $\di$. Define
$$
\widehat{\di}(f,g)=\maxi\big\{\di\big(f(x),g(x)\big):x\in X\big\}
$$
for $f,g\in\CC(X,Y)$. It is well-known that $\widehat{\di}$ is a metric on $\CC(X,Y)$, and the topology that it induces is known as the \emph{topology of uniform convergence}. It is clear that a sequence of continuous functions converges uniformly if and only if it converges with respect to this topology.

Let $X$ and $Y$ be spaces. Given a compact subset $K$ of $X$ and an open subset $U$ of $Y$, set
$$
[K,U]=\{f\in\CC(X,Y):f[K]\subseteq U\}.
$$
The \emph{compact-open topology} on $\CC(X,Y)$ is the one generated by the sets $[K,U]$. This topology will be used explicitly only in the proof of Theorem \ref{theorem_gpp_compact}. In the rest of this article, we will only care about two facts: the first is that it is defined in purely topological terms (that is, no metric is involved), and the second is given by the following well-known result (which follows immediately from \cite[Theorems 46.2 and 46.8]{munkres}; see also \cite[Proposition 1.3.3]{van_mill_book}).

\begin{theorem}\label{theorem_compact_open}
Let $X$ be a compact space, and let $Y$ be a metric space. Then the topology of uniform convergence on $\CC(X,Y)$ coincides with the compact-open topology.
\end{theorem}

\section{The inductive convergence criterion}\label{section_convergence}

In the proof of Theorem \ref{theorem_main}, we will construct various sequences of homeomorphisms (one for each coordinate). To make sure that all of these sequences converge to homeomorphisms, we will apply the so-called \emph{Inductive Convergence Criterion}, as given by Theorem \ref{theorem_convergence}. The intuition behind this criterion is that an ``infinite composition'' of homeomorphisms $\cdots\circ h_1\circ h_0$ will yield a homeomorphism if each $h_{n+1}$ causes ``very small changes'' with respect to a complete metric on the underlying space. One important thing to notice about this result is that the first homeomorphism in the sequence can be chosen without any restriction.

Several versions of this criterion can be found in the literature, but the earliest one seems to be due to Anderson and Bing \cite[Theorem 4.2]{anderson_bing} (see also \cite[Lemma 5.1]{anderson_curtis_van_mill}, or \cite[Exercise 1.6.1 and its solution on pages 533-534]{van_mill_book}). Although all these results accomplish more or less the same thing, here we will follow the approach of \cite[page 195]{yang}, since it is the most direct (that is, it makes no mention of open covers). However, we feel that the proof given in \cite{yang} does not give enough details as to why $H^\ast$ should be the inverse of $H$ (these functions are defined exactly as in the proof of Theorem \ref{theorem_convergence}).\footnote{\,To be fair, mutatis mutandis, neither does the proof given in \cite{anderson_bing}. In fact, we suspect that this might be the reason for the different approach taken in \cite{anderson_curtis_van_mill}.} Therefore, we attempted to be completely rigorous on this point.

Let $X$ be a metric space with distance $\di$, let $x_{m,n}\in X$ for $m,n\in\omega$, and let $x\in X$. We will write $x_{m,n}\to x$ to mean that for every $\varepsilon >0$ there exists $K\in\omega$ such that $\di(x_{m,n},x)<\varepsilon$ whenever $m,n\geq K$. The usual argument about the uniqueness of limits shows that if $x_{m,n}\to x$ and $x_{m,n}\to x'$ then $x=x'$.

\begin{theorem}\label{theorem_convergence}
Let $X$ be a complete metric space with distance $\di$, and let $h_n\in\HH(X)$ for $n\in\omega$. Set $H_n=h_n\circ\cdots\circ h_0$ for $n\in\omega$. Assume that the following conditions hold for each $n$:
\begin{enumerate}
\item $\di\big(h_{n+1}(x),x\big)\leq 2^{-n}$ for every $x\in X$,
\item $\di\big(H_{n+1}^{-1}(x),H_n^{-1}(x)\big)\leq 2^{-n}$ for every $x\in X$.
\end{enumerate}
Then $(H_n:n\in\omega)$ converges uniformly to an element of $\HH(X)$.
\end{theorem}
\begin{proof}
Proposition \ref{proposition_convergence} and condition $(1)$ ensure that $(H_n:n\in\omega)$ converges uniformly to some $H\in\CC(X,X)$. Similary, Proposition \ref{proposition_convergence} and condition $(2)$ ensure that $(H^{-1}_n:n\in\omega)$ converges uniformly to some $H^\ast\in\CC(X,X)$. It remains to show that $H^\ast$ is the inverse of $H$. So fix $x\in X$, then set
$$
x_{m,n}=H_m^{-1}\big(H_n(x)\big)\text{ and }y_{m,n}=H_m\big(H_n^{-1}(x)\big).
$$
for $m,n\in\omega$. Clearly, the following four claims will conclude the proof.

\noindent\textbf{Claim 1.} $x_{m,n}\to x$.

\noindent\textit{Proof.} Pick $\varepsilon >0$. Fix $K\in\omega$ such that $2^{-K}+2^{-(K+1)}+\cdots <\varepsilon$. Assume that $m>n\geq K$, and set $y=H_n(x)$. Then
\begin{align}
\di(x_{m,n},x)&= \di\big(H_m^{-1}(y),H_n^{-1}(y)\big)\nonumber\\
&\leq\di\big(H_m^{-1}(y),H_{m-1}^{-1}(y)\big)+\cdots +\di\big(H_{n+1}^{-1}(y),H_n^{-1}(y)\big)\nonumber\\
&\leq 2^{-(m-1)}+\cdots +2^{-n}\nonumber\\
&<\varepsilon\nonumber
\end{align}
by condition $(2)$. A similar argument yields the same inequality when $n>m\geq K$, and the case $m=n\geq K$ is trivial. $\blacksquare$

\noindent\textbf{Claim 2.} $x_{m,n}\to H^\ast\big(H(x)\big)$.

\noindent\textit{Proof.} Pick $\varepsilon >0$. Since $H^\ast$ is continuous at $H(x)$, we can fix $\delta >0$ such that
$$
\di\Big(H^\ast\big(H(x)\big),H^\ast(y)\Big)<\frac{\varepsilon}{2}\text{ whenever }y\in X\text{ and }\di\big(H(x),y\big)<\delta.
$$
Since $(H_n:n\in\omega)$ converges to $H$, we can fix $N\in\omega$ such that
$$
\di\big(H_n(x),H(x)\big)<\delta\text{ whenever }n\geq N.
$$
Since $(H^{-1}_m:m\in\omega)$ converges uniformly to $H^\ast$, we can fix $M\in\omega$ such that
$$
\di\big(H^{-1}_m(y),H^\ast(y)\big)<\frac{\varepsilon}{2}\text{ whenever }m\geq M\text{ and }y\in X.
$$
Finally, set $K=\maxi\{M,N\}$. Then
\begin{align}
\di\Big(x_{m,n},H^\ast\big(H(x)\big)\Big)&=\di\Big(H_m^{-1}\big(H_n(x)\big),H^\ast\big(H(x)\big)\Big)\nonumber\\
&\leq\di\Big(H_m^{-1}\big(H_n(x)\big),H^\ast\big(H_n(x)\big)\Big)+\di\Big(H^\ast\big(H_n(x)\big),H^\ast\big(H(x)\big)\Big)\nonumber\\
&<\frac{\varepsilon}{2}+\frac{\varepsilon}{2}\nonumber\\
&=\varepsilon\nonumber
\end{align}
whenever $m,n\geq K$. $\blacksquare$

\noindent\textbf{Claim 3.} $y_{m,n}\to x$.

\noindent\textit{Proof.} Proceed as in the proof of Claim 1, except that condition $(1)$ should be used instead of condition $(2)$. $\blacksquare$

\noindent\textbf{Claim 4.} $y_{m,n}\to H\big(H^\ast(x)\big)$.

\noindent\textit{Proof.} Proceed as in the proof of Claim 2. $\blacksquare$
\end{proof}

\section{General position: general facts}\label{section_general}

As it often happens with mathematical ideas, the broad concept of general position can be specified in many different ways, depending on the context. The following definition is the one that will be useful in the present context. Ultimately, in the proof of Theorem \ref{theorem_main}, general position will permit us to construct the desired homeomorphism of a product by working coordinate-wise (so that it will be possible to apply Theorem \ref{theorem_convergence}).

\begin{definition}\label{definition_general_position}
Let $\kappa$ be a cardinal, and let $X_\alpha$ for $\alpha\in\kappa$ be spaces. Set $X=\prod_{\alpha\in\kappa}X_\alpha$. A set $D\subseteq X$ is in \emph{general position} (or is \emph{very thin}) if
$$
\{\alpha\in\kappa:d(\alpha)\neq e(\alpha)\}=\kappa\text{ for every }\{d,e\}\in [D]^2.
$$
We will say that $X$ has the \emph{general position property} ($\GPP$) if for every countable $D\subseteq X$ there exists $h\in\HH(X)$ such that $h[D]$ is in general position.
\end{definition}

Trivially, every product of $\kappa$ factors has the $\GPP$ if $\kappa\leq 1$. Furthermore, it should be clear that the $\GPP$ is not a topological property, but it depends on the product structure. For example, set $X=2\times 2^\omega$ and $Y=2^\omega$, where we view $Y$ as a product with only one factor. It is obvious that $X$ and $Y$ are homeomorphic, but no subset of $X$ containing more than $2$ points is in general position.

The $\GPP$ will be one of the items in the ``checklist'' that guarantees the countable dense homogeneity of a product (see Theorem \ref{theorem_main}). In fact, the main results of this section are Lemmas \ref{lemma_wgpp} and \ref{lemma_gpp}, which reduce the $\GPP$ of a large product to the $\GPP$ of certain countably infinite subproducts, plus the existence of certain pairs of functions (see Definition \ref{definition_convenient_pair}).

We begin, however, with two more fundamental facts. Proposition \ref{proposition_cdh_implies_gpp} shows that, under rather mild assumptions, the $\GPP$ is a weakening of countable dense homogeneity; this proposition follows easily from Lemma \ref{lemma_exists_gp}. According to \cite{gruenhage_natkaniec_piotrowski}, the first part of Lemma \ref{lemma_exists_gp} is a particular case of a result from \cite{piotrowski}. The second part of Lemma \ref{lemma_exists_gp} will be useful in the proof of Lemma \ref{lemma_gpp_manifolds_auxiliary}, where we will need to ``chase'' countable sets away from the pseudoboundary. Recall that a \emph{$\pi$-base} (or, in the terminology of \cite{oxtoby}, a \emph{pseudobase}) for a space $X$ is a collection $\BB$ of non-empty open subsets of $X$ such that for every non-empty open subset $U$ of $X$ there exists $V\in\BB$ such that $V\subseteq U$.

\begin{lemma}\label{lemma_exists_gp}
Let $\kappa\leq\omega$ be a cardinal, and let $X_\alpha$ for $\alpha\in\kappa$ be spaces. Assume that each $X_\alpha$ is crowded and has a countable $\pi$-base. Set $X=\prod_{\alpha\in\kappa}X_\alpha$. Then there exists a countable dense subset $D$ of $X$ that is in general position. Furthermore, if each $X_\alpha$ is also Hausdorff and Baire, and $M$ is a meager subset of $X$, then it is possible to ensure that $D\cap M=\varnothing$.
\end{lemma}
\begin{proof}
Let $\{U_n:n\in\omega\}$ be a $\pi$-base for $X$, and assume without loss of generality that each $U_n=\prod_{\alpha\in\kappa}U_{n,\alpha}$ for suitable (non-empty) open subsets $U_{n,\alpha}$ of $X_\alpha$. Pick an arbitrary $d_0\in U_0$. Given $1\leq n<\omega$ and $d_k\in U_k$ for $k<n$, define $d_n\in U_n$ by picking
$$
d_n(\alpha)\in U_{n,\alpha}\setminus\{d_k(\alpha):k<n\}
$$
for $\alpha\in\kappa$; this is possible since each $X_\alpha$ is crowded. Setting $D=\{d_n:n\in\omega\}$ will then conclude the proof of the first part of the lemma.

In order to prove the second part, assume that each $X_\alpha$ is Hausdorff and Baire, and let $M$ be a meager subset of $X$. Make the following definitions:
\begin{itemize}
\item $M_0=\{x\in X^\omega:\{x(n):n\in\omega\}\text{ is not dense in }X\}$,
\item $M_1=\bigcup_{\{m,n\}\in [\omega]^2}\bigcup_{\alpha\in\kappa}\{x\in X^\omega:x(m)(\alpha)=x(n)(\alpha)\}$,
\item $M_2=\bigcup_{n\in\omega}\{x\in X^\omega:x(n)\in M\}$.
\end{itemize}
Using the assumption that each $X_\alpha$ (hence $X$) has a countable $\pi$-base, it is easy to see that $M_0$ is meager. On the other hand, using the assumption that each $X_\alpha$ is Hausdorff and crowded, one can show that $M_1$ is meager. Finally, it is clear that $M_2$ is meager. Since $X^\omega$ is a product of Baire spaces with countable $\pi$-bases, it is a Baire space by \cite[Theorem 3]{oxtoby}. Therefore, it is possible to pick $x\in X^\omega\setminus (M_0\cup M_1\cup M_2)$. It is clear that $D=\{x(n):n\in\omega\}$ will be as desired.
\end{proof}

\begin{proposition}\label{proposition_cdh_implies_gpp}
Let $\kappa\leq\omega$ be a cardinal, and let $X_\alpha$ for $\alpha\in\kappa$ be spaces. Assume that each $X_\alpha$ is crowded and has a countable $\pi$-base. Set $X=\prod_{\alpha\in\kappa}X_\alpha$. If $X$ is $\CDH$ then $X$ has the $\GPP$.
\end{proposition}
\begin{proof}
Use the first part of Lemma \ref{lemma_exists_gp}.
\end{proof}

By considering the example $X$ given after Definition \ref{definition_general_position}, one sees that the assumption that each $X_\alpha$ is crowded cannot be dropped in either of the above results. Furthermore, according to \cite{gruenhage_natkaniec_piotrowski}, it was proved independently by Schr\"{o}der \cite{schroder} and Szeptycki \cite{szeptycki} that there exists a countable crowded space $X$ such that $X^2$ has no dense subset in general position (in fact, as shown in \cite[Lemma 2.1]{hutchison_gruenhage}, this example can be realized as a countable dense subpace of $2^\cccc$). For other results concerning the $\GPP$, see Theorem \ref{theorem_gpp_compact}, \cite[Theorem 4.5]{dobrowolski_krupski_marciszewski}, \cite[Section 2]{gruenhage_natkaniec_piotrowski}, \cite[Theorem 2.3]{hernandez_gutierrez_arrow} and \cite{medini}.\footnote{\,In some cases, one has to inspect the proofs to realize that the $\GPP$ is involved.}

Next, we introduce the auxiliary concepts of weak general position and convenient pair. As Lemma \ref{lemma_wgpp} will show, the existence of convenient pairs is useful in bringing sets in weak general position. At that point, as Lemma \ref{lemma_gpp} will show, obtaining general position is fairly straightforward.

\begin{definition}\label{definition_weak_general_position}
Let $\kappa$ be a cardinal, and let $X_\alpha$ for $\alpha\in\kappa$ be spaces. Set $X=\prod_{\alpha\in\kappa}X_\alpha$. A set $D\subseteq X$ is in \emph{weak general position} if
$$
|\{\alpha\in\kappa:d(\alpha)\neq e(\alpha)\}|=\kappa\text{ for every }\{d,e\}\in [D]^2.
$$
We will say that $X$ has the \emph{weak general position property} ($\WGPP$) if for every countable $D\subseteq X$ there exists $h\in\HH(X)$ such that $h[D]$ is in weak general position.
\end{definition}

It is clear that every set in general position is in weak general position, and that the two notions are equivalent in the case of finite products.

\begin{definition}\label{definition_convenient_pair}
Let $X$ and $Y$ be spaces. We will say that $(s,t)$ is a \emph{convenient pair for $(X,Y)$} if the following conditions hold:
\begin{itemize}
\item $s,t\in\CC(X\times Y,X)$,
\item $s\big(t(x,y),y\big)=x$ for all $(x,y)\in X\times Y$,
\item $t\big(s(x,y),y\big)=x$ for all $(x,y)\in X\times Y$,
\end{itemize}
Given $A\subseteq X$ and $B\subseteq Y$, we will say that convenient pair $(s,t)$ for $(X,Y)$ is \emph{focused on $(A,B)$} if the following additional condition holds:
\begin{itemize}
\item If $x\in A$ and $y,y'\in B$ are such that $y\neq y'$, then $s(x,y)\neq s(x,y')$.
\end{itemize}
\end{definition}

Notice that letting $s=t=\pi$, where $\pi$ is the projection on the first coordinate, gives a trivial example of convenient pair. However, this convenient pair will only be focused on $(A,B)$ for trivial choices of $(A,B)$, hence it would not be possible to apply Lemma \ref{lemma_wgpp} to it. So, one might want to look for convenient pairs that are ``perturbations'' of the projection. This is the intuition behind the results of Section \ref{section_manifolds}, where we will construct non-trivial convenient pairs in the context of manifolds with boundary. On the other hand, as we will see in Section \ref{section_groups}, non-trivial convenient pairs are readily available in the context of topological groups.

After constructing suitable convenient pairs, the following two lemmas will allow us to obtain the desired results (see Theorems \ref{theorem_gpp_groups} and \ref{theorem_gpp_manifolds}).

\begin{lemma}\label{lemma_wgpp}
Let $\kappa$ be an infinite cardinal, and let $X_\alpha$ for $\alpha\in\kappa$ be spaces. Set $X=\prod_{\alpha\in\kappa}X_\alpha$, and let $D\subseteq X$ be countable. Assume that the following conditions hold:
\begin{itemize}
\item $\pi_0\re D$ is injective,
\item For every $\alpha\in\kappa\setminus\{0\}$ there exists a convenient pair for $(X_\alpha,X_0)$ focused on $(\pi_\alpha[D],\pi_0[D])$.
\end{itemize}
Then there exists $h\in\HH(X)$ such that $h[D]$ is in weak general position.
\end{lemma}
\begin{proof}
Given $p=\{d,e\}\in[D]^2$, define
$$
\Omega_p=\{\alpha\in\kappa:d(\alpha)\neq e(\alpha)\}.
$$
Also set $P=\{p\in [D]^2:|\Omega_p|=\kappa\}$.

\noindent\textbf{Claim.} There exists $\Omega\in [\kappa]^\kappa$ such that $|\Omega_p\setminus\Omega|=\kappa$ for every $p\in P$.

\noindent\textit{Proof.} If $P=\varnothing$ then set $\Omega=\kappa$. Now assume that $P\neq\varnothing$. Begin by fixing pairwise disjoint $\Omega'_p\in [\Omega_p]^\kappa$ for $p\in P$. Then choose $\Omega^0_p,\Omega^1_p\in [\Omega'_p]^\kappa$ for $p\in P$ such that $\Omega^0_p\cap\Omega^1_p=\varnothing$. Finally, set $\Omega=\bigcup_{p\in P}\Omega^0_p$. It is clear that $\Omega$ will be as desired. $\blacksquare$

Let $\Omega$ be as in the above claim, and assume without loss of generality that $0\notin\Omega$. By assumption, for every $\alpha\in\Omega$ we can fix a convenient pair $(s_\alpha,t_\alpha)$ for $(X_\alpha,X_0)$ focused on $(\pi_\alpha[D],\pi_0[D])$. Given $x\in X$ and $\alpha\in\kappa$, make the following definitions:
$$
\left.
\begin{array}{lcl}
& & h(x)(\alpha)= \left\{
\begin{array}{ll}
x(\alpha) & \textrm{if }\alpha\notin\Omega,\\
s_\alpha\big(x(\alpha),x(0)\big) & \textrm{if }\alpha\in\Omega,
\end{array}
\right.
\end{array}
\right.
$$
$$
\left.
\begin{array}{lcl}
& & h^\ast(x)(\alpha)= \left\{
\begin{array}{ll}
x(\alpha) & \textrm{if }\alpha\notin\Omega,\\
t_\alpha\big(x(\alpha),x(0)\big) & \textrm{if }\alpha\in\Omega.
\end{array}
\right.
\end{array}
\right.
$$
Using the fact that each $(s_\alpha,t_\alpha)$ is a convenient pair, one sees that $h,h^\ast\in\CC(X,X)$ and $h^\ast=h^{-1}$. Hence $h\in\HH(X)$.

It remains to show that $h[D]$ is in weak general position. Clearly, this is equivalent to showing that $|\{\alpha\in\kappa:h(d)(\alpha)\neq h(e)(\alpha)\}|=\kappa$ for every $\{d,e\}\in [D]^2$. So pick $p=\{d,e\}\in [D]^2$. First assume that $p\notin P$. Observe that $d(0)\neq e(0)$ by the injectivity of $\pi_0\upharpoonright D$. Therefore
$$
h(d)(\alpha)=s_\alpha\big(d(\alpha),d(0)\big)\neq s_\alpha\big(e(\alpha),e(0)\big)=h(e)(\alpha)
$$
for every $\alpha\in\Omega\setminus\Omega_p$, because $(s_\alpha,t_\alpha)$ is focused on $(\pi_\alpha[D],\pi_0[D])$ and $d(\alpha)=e(\alpha)$ for these values of $\alpha$. But $|\Omega\setminus\Omega_p|=\kappa$ because $|\Omega|=\kappa$ and $p\notin P$, so the proof is concluded in this case. Finally, assume that $p\in P$. In this case, it suffices to observe that
$$
h(d)(\alpha)=d(\alpha)\neq e(\alpha)=h(e)(\alpha)
$$
for every $\alpha\in\Omega_p\setminus\Omega$, since this set has size $\kappa$ by our choice of $\Omega$.
\end{proof}

An early version of the following result only covered the case in which $\Omega^\ast=\varnothing$; the more general version given below will be needed in the proof of Theorem \ref{theorem_gpp_manifolds}.

\begin{lemma}\label{lemma_gpp}
Let $\kappa$ be a cardinal, let $X_\alpha$ for $\alpha\in\kappa$ be spaces, and let $\Omega^\ast\subseteq\kappa$ be either empty or infinite. Set $X=\prod_{\alpha\in\kappa}X_\alpha$. Assume that the following conditions hold:
\begin{itemize}
\item $X$ has the $\WGPP$,
\item $X_\Omega$ has the $\GPP$ for every $\Omega\in [\kappa]^\omega$ such that $\Omega\cap\Omega^\ast$ is empty or infinite.
\end{itemize}
Then $X$ has the $\GPP$.
\end{lemma}
\begin{proof}
The lemma is trivial if $\kappa\leq\omega$, so assume that $\kappa$ is uncountable. Set $\lambda=|\Omega^\ast|$, and fix $\Omega^\ast_\alpha\in [\Omega^\ast]^\omega$ for $\alpha\in\lambda$ such that $\Omega^\ast_\alpha\cap\Omega^\ast_\beta=\varnothing$ whenever $\alpha\neq\beta$ and $\bigcup_{\alpha\in\lambda}\Omega^\ast_\alpha=\Omega^\ast$. Let $D$ be a countable subset of $X$. Since $X$ has the $\WGPP$, we can assume without loss of generality that $D$ is in weak general position. Using this assumption together with the uncountability of $\kappa$, one can easily find $\Omega_\alpha\in [\kappa]^\omega$ for $\alpha\in\kappa$ such that the following conditions are satisfied:
\begin{enumerate}
\item $\Omega_\alpha\cap\Omega_\beta=\varnothing$ whenever $\alpha\neq\beta$,
\item $\alpha\in\bigcup_{\beta\leq\alpha}\Omega_\beta$ for every $\alpha\in\kappa$,
\item Each $\tau_{\Omega_\alpha}\upharpoonright D$ is injective,
\item If $\Omega_\alpha\cap\Omega^\ast_\beta\neq\varnothing$ for some $\alpha\in\kappa$ and $\beta\in\lambda$ then $\Omega^\ast_\beta\subseteq\Omega_\alpha$.
\end{enumerate}

Notice that each $\Omega_\alpha\cap\Omega^\ast_\beta$ is either empty or infinite by condition $(4)$. Therefore, each $X_{\Omega_\alpha}$ has the $\GPP$. Pick $h_\alpha\in\HH(X_{\Omega_\alpha})$ for $\alpha\in\kappa$ such that each $h_\alpha\big[\tau_{\Omega_\alpha}[D]\big]$ is in general position (in $X_{\Omega_\alpha}$), then set $h=\prod_{\alpha\in\kappa}h_\alpha$. Since condition $(1)$ holds and $\bigcup_{\alpha\in\kappa}\Omega_\alpha=\kappa$ by condition $(2)$, one sees that $h\in\HH(X)$. Finally, using condition~$(3)$, it is straightforward to verify that $h[D]$ is in general position.
\end{proof}

\section{General position: topological groups}\label{section_groups}

The purpose of this section is to obtain Theorem \ref{theorem_gpp_groups}, which will allow us to apply Theorem \ref{theorem_main} to certain topological groups (see the proof of Corollary \ref{corollary_main_zero_dimensional}). This section can also be viewed as a warm-up to Section \ref{section_manifolds}, where our strategy will be roughly the same, although saddled by a myriad of technical complications.

\begin{lemma}\label{lemma_convenient_pair_groups}
Let $X$ be a topological group, and let $Y$ be a subspace of $X$. Then there exists a convenient pair for $(X,Y)$ focused on $(X,Y)$.	
\end{lemma}
\begin{proof}
Simply define $s,t:X\times Y\longrightarrow X$ by setting
\begin{itemize}
\item $s(x,y)=x\cdot y$,
\item $t(x,y)=x\cdot y^{-1}$.
\end{itemize}
It is trivial to check that $(s,t)$ is as desired.
\end{proof}

\begin{theorem}\label{theorem_gpp_groups}
Let $\kappa$ be an infinite cardinal, let $X_0$ be a space, and let $X_\alpha$ for $\alpha\in\kappa\setminus\{0\}$ be topological groups. Set $X=\prod_{\alpha\in\kappa}X_\alpha$. Assume that the following conditions hold:
\begin{itemize}
\item $X_0$ is a subspace of $X_\alpha$ for every $\alpha\in\kappa\setminus\{0\}$,
\item $X_\Omega$ has the $\GPP$ for every $\Omega\in [\kappa]^\omega$.
\end{itemize}
Then $X$ has the $\GPP$.
\end{theorem}
\begin{proof}
By applying Lemma \ref{lemma_gpp} with $\Omega^\ast=\varnothing$, it will be enough to show that $X$ has the $\WGPP$. So pick a countable $D\subseteq X$. Fix $\Omega\in [\kappa]^\omega$ such that $0\in\Omega$ and $\tau_\Omega\re D$ is injective. Since $X_\Omega$ has the $\GPP$, we can actually assume that $\pi_0\re D$ is injective. Furthermore, by Lemma \ref{lemma_convenient_pair_groups}, for every $\alpha\in\kappa\setminus\{0\}$ we can fix a convenient pair for $(X_\alpha,X_0)$ focused on $(X_\alpha,X_0)$ (hence on $(\pi_\alpha[D],\pi_0[D])$). At this point, it remains to apply Lemma \ref{lemma_wgpp}.
\end{proof}

\section{General position: manifolds with boundary}\label{section_manifolds}

The purpose of this section is to show that, under rather mild assumptions, a product of manifolds with boundary has the $\GPP$ (see Theorem \ref{theorem_gpp_manifolds}). The strategy of the proof is to apply Lemmas \ref{lemma_wgpp} and \ref{lemma_gpp}, where the $\GPP$ of countably infinite subproducts will be given by Theorem \ref{theorem_yang}. To this end, we will construct convenient pairs for discs (equivalently, for the closures of suitable open subspaces of the given manifolds), then glue them together. To ensure that the gluing process will yield a continuous function, we will need to be able to choose these convenient pairs so that they are arbitrarily close to the projection on the first coordinate. For this reason, we will construct a uniformly convergent sequence of convenient pairs instead of a single one. Throughout this section, we will assume that $1\leq n<m<\omega$.

We begin with several auxiliary definitions. Set
$$
\SSS^n_\ast=\{(x/2)-(1/2,0,\ldots,0):x\in\SSS^n\},
$$
and observe that $(0,\ldots,0)\in\SSS^n_\ast$ and that $||v||\leq 1$ for every $v\in\SSS^n_\ast$. By ``wrapping the disc around the sphere,'' it is possible to obtain a continuous $\psi:\DDD^n\longrightarrow\SSS^n_\ast$ such that $\psi\re\BBB^n$ is injective and $\psi\re\SSS^{n-1}$ is constant with value $(0,\ldots,0)$ (see \cite[Example 7.1.7]{singh} for an explicit expression). Since $n<m$, we can let $i:\RRR^{n+1}\longrightarrow\RRR^m$ be the embedding that appends $m-(n+1)$ zeros to each vector of $\RRR^{n+1}$. In conclusion, by setting $\varphi=i\circ\psi$, we obtain a continuous $\varphi:\DDD^n\longrightarrow\RRR^m$ with the following properties:
\begin{itemize}
\item $\varphi\re\BBB^n$ is injective,
\item $\varphi\re\SSS^{n-1}$ is constant with value $(0,\ldots,0)$,
\item $||\varphi(y)||\leq 1$ for every $y\in\DDD^n$.
\end{itemize}

Given $k\in\omega$, define $f_k,g_k:\RRR^m\times\DDD^n\longrightarrow\RRR^m$ by setting
\begin{itemize}
\item $f_k(x,y)=x+2^{-k}\cdot\varphi(y)$,
\item $g_k(x,y)=x-2^{-k}\cdot\varphi(y)$.
\end{itemize}
Set $h(x)=x/(1-||x||)$ for $x\in\BBB^m$, and observe that $h:\BBB^m\longrightarrow\RRR^m$ is a homeomorphism whose inverse is described by $h^{-1}(x)=x/(1+||x||)$ for $x\in\RRR^m$.

Given $k\in\omega$, define $s_k,t_k:\DDD^m\times\DDD^n\longrightarrow\DDD^m$ by setting
$$
s_k(x,y)=
\left\{
\begin{array}{ll}
h^{-1}\Big(f_k\big(h(x),y\big)\Big) & \textrm{if }(x,y)\in\BBB^m\times\DDD^n,\\
x & \textrm{if }(x,y)\in\SSS^{m-1}\times\DDD^n,
\end{array}
\right.
$$
$$
t_k(x,y)=
\left\{
\begin{array}{ll}
h^{-1}\Big(g_k\big(h(x),y\big)\Big) & \textrm{if }(x,y)\in\BBB^m\times\DDD^n,\\
x & \textrm{if }(x,y)\in\SSS^{m-1}\times\DDD^n.
\end{array}
\right.
$$

We begin by showing that the functions $s_k$ and $t_k$ are continuous. In fact, we will need to prove something stronger (see Lemma \ref{lemma_equicontinuous}). For this purpose, the following two calculus exercises will be useful.

\begin{lemma}\label{lemma_convergence_easy}
Let $x_0\in\SSS^{m-1}$, and let $\varepsilon>0$. Then there exist $R>0$ and $\delta>0$ such that the following holds for every $z\in\RRR^m$:
$$
||h^{-1}(z)-x_0||<\varepsilon\text{ whenever }||z||>R\text{ and }\left|\left|\frac{z}{||z||}-x_0\right|\right|<\delta.
$$
\end{lemma}
\begin{proof}
Simply observe that
$$
h^{-1}(z)=\frac{z}{||z||}\cdot\frac{1}{\frac{1}{||z||}+1}
$$
for every $z\in\RRR^m\setminus\{(0,\ldots,0)\}$.
\end{proof}

\begin{lemma}\label{lemma_convergence_hard}
Let $x_0\in\SSS^{m-1}$, and let $\varepsilon>0$. Then there exist $R>0$ and $\delta>0$ such that the following holds for every $u,v\in\RRR^m$:
$$
\left|\left|\frac{u+v}{||u+v||}-x_0\right|\right|<\varepsilon\text{ whenever }||u||>R\text{, }||v||\leq 1\text{ and }\left|\left|\frac{u}{||u||}-x_0\right|\right|<\delta.
$$
\end{lemma}
\begin{proof}
Start by choosing $R'\geq 1$ such that
$$
\frac{||v||}{||u+v||}\leq\frac{1}{||u||-1}<\frac{\varepsilon}{3}
$$
whenever $||u||>R'$ and $||v||\leq 1$. Notice that
$$
\frac{1}{1+\frac{1}{||u||}}=\frac{||u||}{||u||+1}\leq\frac{||u||}{||u+v||}\leq\frac{||u||}{||u||-1}=\frac{1}{1-\frac{1}{||u||}}
$$
whenever $||u||>1$ and $||v||\leq 1$. Therefore, we can find $R\geq R'$ such that
$$
\left|\left|\frac{u}{||u+v||}-\frac{u}{||u||}\right|\right|=\left|\frac{||u||}{||u+v||}-1\right|<\frac{\varepsilon}{3}
$$
whenever $||u||>R$ and $||v||\leq 1$. Finally, set $\delta=\varepsilon/3$. Since
$$
\left|\left|\frac{u+v}{||u+v||}-x_0\right|\right|\leq\frac{||v||}{||u+v||}+\left|\left|\frac{u}{||u+v||}-\frac{u}{||u||}\right|\right|+\left|\left|\frac{u}{||u||}-x_0\right|\right|,
$$
it is clear that $R$ and $\delta$ are as desired.
\end{proof}

\begin{lemma}\label{lemma_equicontinuous}
Set $\Ss=\{s_k:k\in\omega\}$ and $\TT=\{t_k:k\in\omega\}$. Then $\Ss$ and $\TT$ are equicontinuous.
\end{lemma}
\begin{proof}
We will only show that $\Ss$ is equicontinuous, since the proof that $\TT$ is equicontinuous is almost identical. So pick $(x_0,y_0)\in\DDD^m\times\DDD^n$. First assume that $x_0\in\BBB^m$. It is straightforward to check that the sequence $(f_k:k\in\omega)$ converges uniformly to the projection on the first coordinate $\pi:\RRR^m\times\DDD^n\longrightarrow\RRR^m$. In particular, this sequence converges uniformly on $K\times\DDD^n$ for every compact $K\subseteq\RRR^m$. Therefore, by Theorem \ref{theorem_compact_open}, the sequence $(s_k:k\in\omega)$ converges uniformly on $K\times\DDD^n$ for every compact $K\subseteq\BBB^m$. By letting $K$ be the closure of a sufficiently small neighborhood of $x_0$ and applying Proposition \ref{proposition_equicontinuous}, one sees that $\Ss$ is equicontinuous at $(x_0,y_0)$.

To conclude the proof, assume that $x_0\in\SSS^{m-1}$. Fix $\varepsilon>0$. By Lemma \ref{lemma_convergence_easy}, we can fix $R>0$ and $\delta>0$ such that
$$
||h^{-1}(z)-x_0||<\varepsilon\text{ whenever }||z||>R\text{ and }\left|\left|\frac{z}{||z||}-x_0\right|\right|<\delta.
$$
By Lemma \ref{lemma_convergence_hard}, we can fix $R'>0$ and $\delta'>0$ such that
$$
\left|\left|\frac{u+v}{||u+v||}-x_0\right|\right|<\delta\text{ whenever }||u||>R'\text{, }||v||\leq 1\text{ and }\left|\left|\frac{u}{||u||}-x_0\right|\right|<\delta'.
$$
Also assume without loss of generality that $||u+v||>R$ whenever $||u||>R'$ and $||v||\leq 1$. Finally, fix a neighborhood $U$ of $x_0$ in $\DDD^m$ such that $||h(x)||>R'$ for every $x\in U\cap\BBB^m$ and $||(x/||x||)-x_0||<\delta'$ for every $x\in U$. By setting $u=h(x)$, $v=2^{-k}\cdot\varphi(y)$ and $z=u+v$, one can easily check that 
$$
||s_k(x,y)-s_k(x_0,y_0)||=||s_k(x,y)-x_0||<\varepsilon
$$
for every $(x,y)\in U\times\DDD^n$ and $k\in\omega$.
\end{proof}

At this point, it will be easy to obtain the uniform convergence mentioned at the beginning of this section.

\begin{lemma}\label{lemma_equicontinuity}
The sequences $(s_k:k\in\omega)$ and $(t_k:k\in\omega)$ both converge uniformly to the projection on the first coordinate $\pi:\DDD^m\times\DDD^n\longrightarrow\DDD^m$.
\end{lemma}
\begin{proof}
Simply combine Theorem \ref{theorem_dini} and Lemma \ref{lemma_equicontinuous}.
\end{proof}

Next, we will show that each $(s_k,t_k)$ is a convenient pair which behaves nicely on the boundary.

\begin{lemma}\label{lemma_local_convenient_pair}
Let $k\in\omega$. Then $(s_k,t_k)$ is a convenient pair for $(\DDD^m,\DDD^n)$ focused on $(\BBB^m,\BBB^n)$. Furthermore $s_k(x,y)=t_k(x,y)=x$ for every $(x,y)\in\partial(\DDD^m\times\DDD^n)$.
\end{lemma}
\begin{proof}
The continuity requirement follows from Lemma \ref{lemma_equicontinuous}. If $(x,y)\in\SSS^{m-1}\times\DDD^n$ then $s_k(x,y)=t_k(x,y)=x$ by definition. On the other hand, if $(x,y)\in\DDD^m\times\SSS^{n-1}$ then $s_k(x,y)=t_k(x,y)=x$ because $\varphi(y)=0$. To complete the proof that $(s_k,t_k)$ is a convenient pair for $(\DDD^m,\DDD^n)$, pick $(x,y)\in\BBB^m\times\BBB^n$. By unwinding the various definitions, one sees that
\begin{align}
s_k\big(t_k(x,y),y\big)&=s_k\bigg(h^{-1}\Big(g_k\big(h(x),y\big)\Big),y\bigg)\nonumber\\\nonumber
&=h^{-1}\left(f_k\Bigg(h\bigg(h^{-1}\Big(g_k\big(h(x),y\big)\Big)\bigg),y\Bigg)\right)\nonumber\\
&=h^{-1}\bigg(f_k\Big(g_k\big(h(x),y\big),y\Big)\bigg)\nonumber\\
&=h^{-1}\Big(f_k\big(h(x)-2^{-k}\cdot\varphi(y),y\big)\Big)\nonumber\\
&=h^{-1}\big(h(x)-2^{-k}\cdot\varphi(y)+2^{-k}\cdot\varphi(y)\big)\nonumber\\
&=h^{-1}\big(h(x)\big)\nonumber\\
&=x.\nonumber
\end{align}
Similarly, one sees that $t_k\big(s_k(x,y),y\big)=x$ for every $(x,y)\in\BBB^m\times\BBB^n$.

It remains to show that $(s_k,t_k)$ is focused on $(\BBB^m,\BBB^n)$. So pick $x\in\BBB^m$ and $y,y'\in\BBB^n$, and assume that $s_k(x,y)=s_k(x,y')$. Since $h^{-1}$ is injective, we must have
$$
h(x)+2^{-k}\cdot\varphi(y)=h(x)+2^{-k}\cdot\varphi(y'),
$$
hence $\varphi(y)=\varphi(y')$. But $\varphi\re\BBB^n$ is injective by our choice of $\varphi$, therefore $y=y'$, as desired.
\end{proof}

Next, by gluing (local) convenient pairs that are sufficiently close to the projection on the first coordinate, we will obtain a (global) convenient pair for a pair of manifolds with boundary.

\begin{lemma}\label{lemma_global_convenient_pair}
Let $1\leq n<m<\omega$, let $X$ and $Y$ be non-empty manifolds with boundary such that $\dime(X)=m$ and $\dime(Y)=n$, and let $D\subseteq X\times Y$ be countable. Assume that $D\cap\partial(X\times Y)=\varnothing$. Then there exists a convenient pair $(s,t)$ for $(X,Y)$ focused on $(\pi_X[D],\pi_Y[D])$, where we denote by $\pi_X:X\times Y\longrightarrow X$ and $\pi_Y:X\times Y\longrightarrow Y$ the natural projections.
\end{lemma}
\begin{proof}
Throughout this proof, we will use $\cl$ to denote closure in $X$, $Y$ or $X\times Y$ (it will be clear from the context which one). Begin by fixing open subsets $W_k$ of $X\times Y$ for $k\in\omega$ that satisfy the following conditions:
\begin{enumerate}
\item $\cl(W_i)\cap\cl(W_j)=\varnothing$ whenever $i\neq j$,
\item $D\subseteq\bigcup_{k\in\omega}W_k$,
\item Each $W_k=U_k\times V_k$, where $U_k$ is an open subset of $X$ and $V_k$ is an open subset of $Y$,
\item Each $\cl(U_k)$ is homeomorphic to $\DDD^m$ via a homeomorphism that maps $U_k$ onto $\BBB^m$,
\item Each $\cl(V_k)$ is homeomorphic to $\DDD^n$ via a homeomorphism that maps $V_k$ onto $\BBB^n$.
\end{enumerate}
The sets $W_k$ can easily be obtained by recursion, making sure to include at every stage the minimal element of $D$ (in some fixed well-ordering of type $\omega$) that has not been covered yet. Conditions $(4)$ and $(5)$ can be satisfied because of the assumption that $D\cap\partial(X\times Y)=\varnothing$.

Now fix an admissible metric $\di$ on $X$, but recall that, by Theorem \ref{theorem_compact_open}, uniform convergence in $\CC(K,X)$ for a compact subset $K$ of $X\times Y$ will not depend on the choice of metric. Using this observation and Lemma \ref{lemma_local_convenient_pair}, one can obtain $(s_k,t_k)$ for $k\in\omega$ such that the following conditions are satisfied for each $k$:
\begin{enumerate}
\item[(6)] $(s_k,t_k)$ is a convenient pair for $\big(\cl(U_k),\cl(V_k)\big)$,
\item[(7)] $(s_k,t_k)$ is focused on $(U_k,V_k)$,
\item[(8)] $s_k(x,y)=x$ for every $(x,y)\in\cl(W_k)\setminus W_k$,
\item[(9)] $t_k(x,y)=x$ for every $(x,y)\in\cl(W_k)\setminus W_k$,
\item[(10)] $\di\big(s_k(x,y),x\big)\leq 2^{-k}$ for every $(x,y)\in\cl(W_k)$,
\item[(11)] $\di\big(t_k(x,y),x\big)\leq 2^{-k}$ for every $(x,y)\in\cl(W_k)$.
\end{enumerate}

Finally, define $s,t:X\times Y\longrightarrow X$ by setting
$$
s(x,y)=
\left\{
\begin{array}{ll}
s_k(x,y) & \textrm{if }(x,y)\in W_k,\\
x & \textrm{if }(x,y)\notin\bigcup_{k\in\omega}W_k,
\end{array}
\right.
$$
$$
t(x,y)=
\left\{
\begin{array}{ll}
t_k(x,y) & \textrm{if }(x,y)\in W_k,\\
x & \textrm{if }(x,y)\notin\bigcup_{k\in\omega}W_k.
\end{array}
\right.
$$
Using the fact that each $s_k$ is continuous, plus conditions $(1)$, $(8)$ and $(10)$, it is not hard to show that $s$ is continuous. Similarly, one sees that $t$ is continuous. Then, using conditions $(6)$, $(8)$ and $(9)$, it is straightforward to complete the proof that $(s,t)$ is a convenient pair for $(X,Y)$. Finally, by conditions $(2)$ and $(7)$, it is clear that $(s,t)$ is focused on $(\pi_X[D],\pi_Y[D])$.
\end{proof}

In conclusion, by combining the methods of Section \ref{section_general} with the convenient pairs given by Lemma \ref{lemma_global_convenient_pair}, we will obtain the desired general position result for products of manifolds with boundary. However, one last auxiliary result will be needed.

\begin{lemma}\label{lemma_gpp_manifolds_auxiliary}
Let $\kappa$ be a cardinal, and let $X_\alpha$ for $\alpha\in\kappa$ be non-empty manifolds with boundary such that each $\dim(X_\alpha)\geq 1$. Set $X=\prod_{\alpha\in\kappa}X_\alpha$, and let $D\subseteq X$ be countable. Assume that each $X_\alpha$ is connected and that condition $(\partial)$ holds. Then there exists $h\in\HH(X)$ such that $\pi_0\re h[D]$ is injective and $h[D]\cap\partial X=\varnothing$.
\end{lemma}
\begin{proof}
We will assume that $\kappa$ is uncountable, since the proof in the case $\kappa\leq\omega$ is similar but easier. As in the proof of Lemma \ref{lemma_gpp}, with $\Omega^\ast=\{\alpha\in\kappa:\partial X_\alpha\neq\varnothing\}$, one can obtain $\Omega_\alpha\in [\kappa]^\omega$ for $\alpha\in\kappa$ such that the following conditions hold:
\begin{itemize}
\item $\Omega_\alpha\cap\Omega_\beta=\varnothing$ whenever $\alpha\neq\beta$,
\item $\bigcup_{\alpha\in\kappa}\Omega_\alpha$,
\item $0\in\Omega_0$,
\item $\tau_{\Omega_0}\re D$ is injective,
\item Each $X_{\Omega_\alpha}$ satisfies condition $(\partial)$.
\end{itemize}

Since each $X_{\Omega_\alpha}$ is a Baire space, we can fix countable dense subsets $E_\alpha$ of $X_{\Omega_\alpha}$ for $\alpha\in\kappa$ such that each $E_\alpha\cap\partial X_{\Omega_\alpha}=\varnothing$. By Lemma \ref{lemma_exists_gp}, we can also assume without loss of generality that $E_0$ is in general position (in $X_{\Omega_0}$). By Theorem \ref{theorem_yang}, we can obtain homeomorphisms $h_\alpha\in\HH(X_{\Omega_\alpha})$ for $\alpha\in\kappa$ such that each $h_\alpha\big[\tau_{\Omega_\alpha}[D]\big]\subseteq E_\alpha$. It is straightforward to check that $h=\prod_{\alpha\in\kappa}h_\alpha$ will be as desired.
\end{proof}

\begin{theorem}\label{theorem_gpp_manifolds}
Let $\kappa\geq 1$ be a cardinal, and let $X_\alpha$ for $\alpha\in\kappa$ be non-empty manifolds with boundary such that $1\leq\dime(X_0)<\dime(X_\alpha)$ for every $\alpha\in\kappa\setminus\{0\}$. Assume that each $X_\alpha$ is connected and that condition $(\partial)$ holds. Then $\prod_{\alpha\in\kappa}X_{\alpha}$ has the $\GPP$.
\end{theorem}
\begin{proof}
If $\kappa$ is finite, the desired results follows from Proposition \ref{proposition_cdh_implies_gpp}, so assume that $\kappa$ is infinite. Set $X=\prod_{\alpha\in\kappa}X_{\alpha}$ and $\Omega^\ast=\{\alpha\in\kappa:\partial X_\alpha\neq\varnothing\}$. Using Theorem \ref{theorem_yang} and Proposition \ref{proposition_cdh_implies_gpp}, one sees that $X_\Omega$ has the $\GPP$ whenever $\Omega\in [\kappa]^\omega$ is such that $\Omega\cap\Omega^\ast$ is empty or infinite. Therefore, by Lemma \ref{lemma_gpp}, it will be enough to show that $X$ has the $\WGPP$.

So pick a countable $D\subseteq X$. By Lemma \ref{lemma_gpp_manifolds_auxiliary}, we can assume without loss of generality that $\pi_0\re D$ is injective and $D\cap\partial X=\varnothing$. In particular, it is possible to apply Lemma \ref{lemma_global_convenient_pair}, which for every $\alpha\in\kappa\setminus\{0\}$ yields a convenient pair $(s_\alpha,t_\alpha)$ for $(X_\alpha,X_0)$ focused on $(\pi_\alpha[D],\pi_0[D])$. By Lemma \ref{lemma_wgpp}, it follows that there exists $h\in\HH(X)$ such that $h[D]$ is in weak general position.
\end{proof}

\section{The main result and its corollaries}\label{section_main}

In this section, we will finally keep the promise of a ``unified treatment'' made in the abstract. To see that all the hypotheses of Theorem \ref{theorem_main} are indispensable, see the table at the end of Section \ref{section_optimality}. Also notice that the cardinal $\kappa$ in the statement of Theorem \ref{theorem_main} need not be infinite. In fact, this result might be viewed as a ``multivariable'' version of Theorem \ref{theorem_fundamental}. However, in the case $\kappa=1$, the requirement of strong $n$-homogeneity is unnecessary, while the $\GPP$ holds trivially.

\begin{theorem}\label{theorem_main}
Let $\kappa<\pppp$ be a cardinal, and let $X_\alpha$ for $\alpha\in\kappa$ be Polish spaces. Set $X=\prod_{\alpha\in\kappa}X_\alpha$. Assume that the following conditions are satisfied:
\begin{itemize}
\item Each $X_\alpha$ is strongly locally homogeneous,
\item Each $X_\alpha$ is strongly $n$-homogeneous for every $n\in\omega$,
\item $X$ has the $\GPP$.
\end{itemize}
Then $X$ is $\CDH$.
\end{theorem}
\begin{proof}
Clearly, we can assume that $\kappa\geq 1$ and that each $|X_\alpha|\geq 2$. If $\kappa$ is finite, we will also assume that each $X_\alpha$ is crowded; using Proposition \ref{proposition_slh}, it is easy to realize that this will not result in any loss of generality. Fix a pair $(D,E)$ of countable dense subsets of $X$. Since $X$ has the $\GPP$, we can assume without loss of generality that $D\cup E$ is in general position. Set $D_\alpha=\pi_\alpha[D]$ and $E_\alpha=\pi_\alpha[E]$ for $\alpha\in\kappa$. Fix admissible complete metrics $\di_\alpha$ on $X_\alpha$ for $\alpha\in\kappa$. Given $x\in X_\alpha$ and $\varepsilon >0$, we will use the notation $\Ball_\alpha(x,\varepsilon)=\{z\in X:\di_\alpha(z,x)<\varepsilon\}$ for the open ball around $x$ of radius $\varepsilon$.

Using the homogeneity properties of the $X_\alpha$, plus the fact that $D$ and $E$ are countable, it is possible to fix countable subgroups $\HH_\alpha$ of $\HH(X_\alpha)$ for $\alpha\in\kappa$ such that the following conditions are satisfied for each $\alpha$:
\begin{itemize}
\item[(1)] For every finite bijection $\sigma$ such that $\dom(\sigma)\subseteq D_\alpha$ and $\ran(\sigma)\subseteq E_\alpha$ there exists $h\in\HH_\alpha$ such that $\sigma\subseteq h$,
\item[(2)] For every $\varepsilon >0$, $d\in D_\alpha$ and $H\in\HH_\alpha$ there exists $0<\delta <\varepsilon$ such that for every $e\in E_\alpha\cap\Ball_\alpha\big(H(d),\delta\big)$ there exists $h\in\HH_\alpha$ such that $h\big(H(d)\big)=e$ and $\supp(h)\subseteq\Ball_\alpha\big(H(d),\delta\big)$,
\item[(3)] For every $\varepsilon >0$, $e\in E_\alpha$ and $H\in\HH_\alpha$ there exists $0<\delta <\varepsilon$ such that for every $d\in D_\alpha\cap H^{-1}\big[\Ball_\alpha(e,\delta)\big]$ there exists $h\in\HH_\alpha$ such that $h\big(H(d)\big)=e$ and $\supp(h)\subseteq\Ball_\alpha(e,\delta)$.
\end{itemize}

Given $\alpha\in\kappa$, we will say that $s$ is an \emph{$\alpha$-suitable tuple} if $s=(h_0,\ldots, h_n)$ for some $n\in\omega$ and the following conditions hold, where $H_i=h_i\circ\cdots\circ h_0$ for $i\leq n$:
\begin{itemize}
\item[(4)] $h_i\in\HH_\alpha$ for every $i\leq n$,
\item[(5)] $\di_\alpha\big(h_{i+1}(x),x\big)\leq 2^{-i}$ for every $i<n$ and $x\in X_\alpha$,
\item[(6)] $\di_\alpha\big(H_{i+1}^{-1}(x),H_i^{-1}(x)\big)\leq 2^{-i}$ for every $i<n$ and $x\in X_\alpha$.
\end{itemize}
Given such an $s$, we will use the notation $H_s=h_n\circ\cdots\circ h_0$. Similarly, we will say that $S$ is an \emph{$\alpha$-suitable sequence} if $S=(h_i:i\in\omega)$ and $(h_0,\ldots, h_n)$ is an $\alpha$-suitable tuple for every $n\in\omega$. Given such an $S$, by Theorem \ref{theorem_convergence}, the sequence $(h_n\circ\cdots\circ h_0:n\in\omega)$ will converge to an element of $\HH(X_\alpha)$, which we will denote by $H_S$.

Denote by $\PPP$ the set of all triples of the form $p=(F,\zeta,\sigma)=(F^p,\zeta^p,\sigma^p)$ such that the following requirements are satisfied:
\begin{itemize}
\item[(7)] $F\in [\kappa]^{<\omega}$,
\item[(8)] $\zeta$ is a function such that $\dom(\zeta)=F$ and each $\zeta(\alpha)$ is an $\alpha$-suitable tuple,
\item[(9)] $\sigma$ is a finite bijection such that $\dom(\sigma)\subseteq D$ and $\ran(\sigma)\subseteq E$,
\item[(10)] $H_{\zeta(\alpha)}\big(d(\alpha)\big)=\sigma(d)(\alpha)$ for every $\alpha\in F$ and $d\in\dom(\sigma)$.
\end{itemize}
Order $\PPP$ by declaring $q\leq p$ if the following conditions are satisfied:
\begin{itemize}
\item[(11)] $F^q\supseteq F^p$,
\item[(12)] $\zeta^q(\alpha)\supseteq\zeta^p(\alpha)$ for every $\alpha\in F^p$,
\item[(13)] $\sigma^q\supseteq\sigma^p$,
\item[(14)] If $\zeta^p(\alpha)=(h_0,\ldots,h_n)$ and $\zeta^q(\alpha)=(h_0,\ldots,h_m)$ for some $\alpha\in F^p$, then $h_i\Big(H_{\zeta^p(\alpha)}\big(d(\alpha)\big)\Big)=H_{\zeta^p(\alpha)}\big(d(\alpha)\big)$ whenever $n<i\leq m$ and $d\in\dom(\sigma^p)$.
\end{itemize}
Intuitively, condition $(14)$ expresses the ``promise'' made by $p$ that $d(\alpha)$ has reached its final destination (which will be in $E_\alpha$ by condition $(10)$).

Given $\alpha\in\kappa$, $d\in D$ and $e\in E$, make the following definitions:
\begin{itemize}
\item $D^\coord_\alpha=\{p\in\PPP:\alpha\in F^p\}$,
\item $D^\dom_d=\{p\in\PPP:d\in\dom(\sigma^p)\}$,
\item $D^\ran_e=\{p\in\PPP:e\in\ran(\sigma^p)\}$.
\end{itemize}

\noindent\textbf{Claim 1.} Each $D^\coord_\alpha$ is a dense subset of $\PPP$.

\noindent\textit{Proof.} Fix $\alpha\in\kappa$. Pick $p=(F,\zeta,\sigma)\in\PPP$, and assume without loss of generality that $\alpha\notin F$. Since $D\cup E$ is in general position, all $d(\alpha)$ are distinct for $d\in\dom(\sigma)$ and all $e(\alpha)$ are distinct for $e\in\ran(\sigma)$. Therefore, by condition $(1)$, there exists $h\in\HH_\alpha$ such that $h\big(d(\alpha)\big)=\sigma(d)(\alpha)$ for every $d\in\dom(\sigma)$. Set $s=(h)$, and observe that $s$ is an $\alpha$-suitable tuple. Finally, set $q=\big(F\cup\{\alpha\},\zeta\cup\{(\alpha,s)\},\sigma\big)$. It is straightforward to check that $q\in D^\coord_\alpha$ and $q\leq p$. $\blacksquare$

\noindent\textbf{Claim 2.} Each $D^\dom_d$ is a dense subset of $\PPP$.

\noindent\textit{Proof.} Fix $d\in D$. Pick $p=(F,\zeta,\sigma)\in\PPP$, and assume without loss of generality that $d\notin\dom(\sigma)$. Given $\alpha\in F$ with $\zeta(\alpha)=(h_0,\ldots,h_n)$, set $x_\alpha=H_{\zeta(\alpha)}\big(d(\alpha)\big)$, then pick $\varepsilon_\alpha >0$ such that the following requirements are satisfied:
\begin{itemize}
\item[(15)] $\varepsilon_\alpha\leq 2^{-(n+1)}$,
\item[(16)] $\di_\alpha\big(H^{-1}_{\zeta(\alpha)}(x),H^{-1}_{\zeta(\alpha)}(x_\alpha)\big)\leq 2^{-(n+1)}$ for every $x\in\Ball_\alpha(x_\alpha,\varepsilon_\alpha)$,
\item[(17)] $\Ball_\alpha(x_\alpha,\varepsilon_\alpha)\cap\pi_\alpha\big[\ran(\sigma)\big]=\varnothing$.
\end{itemize}
For each $\varepsilon_\alpha$, fix a corresponding $\delta_\alpha$ as given by condition $(2)$, then set
$$
U=\{x\in X:x(\alpha)\in\Ball_\alpha(x_\alpha,\delta_\alpha)\text{ for every }\alpha\in F\}.
$$
If $\kappa$ is infinite, then $U$ is infinite by the assumption that each $|X_\alpha|\geq 2$. On the other hand, if $\kappa$ is finite, then $U$ is infinite by the assumption that $\kappa\geq 1$ and each $X_\alpha$ is crowded. In either case, since $E$ is dense in $X$, it is possible to find $e\in \big(U\setminus\ran(\sigma)\big)\cap E$. Now, using condition $(2)$, fix $h_\alpha\in\HH_\alpha$ for $\alpha\in F$ such that each $h_\alpha(x_\alpha)=e(\alpha)$ and each $\supp(h_\alpha)\subseteq\Ball_\alpha(x_\alpha,\delta_\alpha)$. Set $\zeta'(\alpha)=\zeta(\alpha)^\frown (h_\alpha)$ for $\alpha\in F$. Since conditions $(15)$ and $(16)$ ensure that conditions $(5)$ and $(6)$ will be satisfied, each $\zeta'(\alpha)$ is an $\alpha$-suitable tuple. Finally, set $q=\big(F,\zeta',\sigma\cup\{(d,e)\}\big)$. It is straightforward to check that $q\in D^\dom_d$ and $q\leq p$. $\blacksquare$

\noindent\textbf{Claim 3.} Each $D^\ran_e$ is a dense subset of $\PPP$.

\noindent\textit{Proof.} Since this is analogous to the proof of Claim 2, we will only point out the most significant differences. Begin by setting $x_\alpha=e(\alpha)$. Using this notation, conditions $(15)$, $(16)$ and $(17)$ remain unchanged. Once the suitable $\varepsilon_\alpha$ and $\delta_\alpha$ are found, define
$$
U=\big\{x\in X:x(\alpha)\in H^{-1}_{\zeta(\alpha)}\big[\Ball_\alpha(x_\alpha,\delta_\alpha)\big]\text{ for every }\alpha\in F\big\}.
$$
Finally, pick $d\in\big(U\setminus\dom(\sigma)\big)\cap D$, then use condition $(3)$ to obtain $h_\alpha\in\HH_\alpha$ for $\alpha\in F$ so that each $h_\alpha\Big(H_{\zeta(\alpha)}\big(d(\alpha)\big)\Big)=x_\alpha$ and $\supp(h_\alpha)\subseteq\Ball_\alpha(x_\alpha,\delta_\alpha)$. $\blacksquare$

\noindent\textbf{Claim 4.} $\PPP$ is $\sigma\text{-}\mathrm{centered}$.

\noindent\textit{Proof.} By \cite[Exercise III.2.12]{kunen}, we can fix an \emph{independent family of functions} of size $\kappa$. That is, we can fix functions $f_\alpha:\omega\longrightarrow\omega$ for $\alpha<\kappa$ such that
$$
\{k\in\omega:f_{\alpha_i}(k)=j_i\text{ for every }i\leq n\}\text{ is infinite}
$$
whenever $n\in\omega$, $\alpha_0<\cdots <\alpha_n<\kappa$ and $j_0,\ldots ,j_n\in\omega$. Given $\alpha<\kappa$, fix an enumeration $\{s_{\alpha,j}:j\in\omega\}$ of all $\alpha$-suitable tuples. Given $k\in\omega$ and a finite bijection $\sigma$ such that $\dom(\sigma)\subseteq D$ and $\ran(\sigma)\subseteq E$, define
$$
\PPP(k,\sigma)=\{p\in\PPP:\sigma^p=\sigma\text{ and }\zeta^p(\alpha)=s_{\alpha,f_\alpha(k)}\text{ for every }\alpha\in F^p\}.
$$
It is clear that each $\PPP(k,\sigma)$ is a centered subset of $\PPP$. To see that the union of these sets is $\PPP$, pick $p=(F,\zeta,\sigma)\in\PPP$. Let $\alpha_0<\cdots <\alpha_n$ be such that $F=\{\alpha_0,\ldots,\alpha_n\}$, and let $j_0,\ldots ,j_n\in\omega$ be such that $\zeta(\alpha_i)=s_{\alpha_i,j_i}$ for every $i\leq n$. It is easy to realize that $p\in\PPP(\sigma,k)$ whenever $k$ is such that $f_{\alpha_i}(k)=j_i$ for every $i\leq n$. $\blacksquare$

Observe that the collection
$$
\DD=\{D^\coord_\alpha:\alpha\in\kappa\}\cup\{D^\dom_d:d\in D\}\cup\{D^\ran_e:e\in E\}
$$
consists of dense subsets of $\PPP$ by Claims $1$-$3$, and that $|\DD|<\pppp$ because $\kappa<\pppp$. Therefore, since $\PPP$ is $\sigma$-centered by Claim $4$, it is possible to apply Bell's Theorem\footnote{\,This result was essentially obtained in \cite{bell}. However, the modern statement given in \cite{kunen} is the one that is needed here.} (see \cite[Theorem III.3.61]{kunen}), which guarantees the existence of a filter $G$ on $\PPP$ that intersects every element of $\DD$. Given $\alpha\in\kappa$, set
$$
S_\alpha =\bigcup\{\zeta^p(\alpha):p\in G\},
$$
and observe that each $S_\alpha$ is either an $\alpha$-suitable tuple or an $\alpha$-suitable sequence. In either case, set $h_\alpha=H_{S_\alpha}$. Finally, set $h=\prod_{\alpha\in\kappa}h_\alpha$, and observe that $h\in\HH(X)$. The verification that $h[D]=E$ is straightforward.
\end{proof}

As a first consequence of our main result, we will show that it is ``very easy'' for an infinite product of zero-dimensional Polish spaces to be $\CDH$. Unfortunately, as Proposition \ref{proposition_optimality_compactness} will show, the assumption concerning compactness is unavoidable. Recall that a zero-dimensional space $X$ is \emph{nowhere compact} if every non-empty open subset of $X$ is non-compact.

\begin{corollary}\label{corollary_main_zero_dimensional}
Let $\kappa<\pppp$ be an infinite cardinal, and let $X_\alpha$ for $\alpha\in\kappa$ be zero-dimensional Polish spaces such that each $|X_\alpha|\geq 2$. Assume that each $X_\alpha$ is either compact or nowhere compact. Then $\prod_{\alpha\in\kappa}X_\alpha$ is $\CDH$.
\end{corollary}
\begin{proof}
By considering countably infinite subproducts, we can assume without loss of generality that each $X_\alpha$ is crowded. By the classical characterizations of $2^\omega$ and $\omega^\omega$ (see for example \cite[Theorems 1.5.5 and 1.9.8]{van_mill_book}), it follows that each $X_\alpha$ is homeomorphic to one of these two spaces. Using this observation together with Proposition \ref{proposition_cdh_implies_gpp}, it is easy to realize that the requirements of Theorem \ref{theorem_gpp_groups} will be satisfied, hence $\prod_{\alpha\in\kappa}X_\alpha$ has $\GPP$. To conclude the proof, simply apply Proposition \ref{proposition_zero_dimensional} and Theorem \ref{theorem_main}.
\end{proof}

\begin{corollary}[Stepr\={a}ns, Zhou]\label{corollary_main_cantor}
Let $\kappa<\pppp$ be a cardinal. Then $2^\kappa$ is $\CDH$.
\end{corollary}

\begin{corollary}\label{corollary_main_baire}
Let $\kappa<\pppp$ be a cardinal. Then $\omega^\kappa$ is $\CDH$.
\end{corollary}

Finally, we will apply the results of Section \ref{section_manifolds} to show that ``$\leq\omega$'' can be substituted by ``$<\pppp$'' in the statement of Theorem \ref{theorem_yang}.

\begin{corollary}\label{corollary_main_manifolds}
Let $\kappa<\pppp$ be a cardinal, and let $X_\alpha$ for $\alpha\in\kappa$ be manifolds with boundary. Assume that each $X_\alpha$ is connected and that condition $(\partial)$ holds. Then $\prod_{\alpha\in\kappa}X_\alpha$ is $\CDH$.
\end{corollary}
\begin{proof}
Set $X=\prod_{\alpha\in\kappa}X_\alpha$. By Theorem \ref{theorem_yang}, we can assume that $\kappa$ is uncountable. Clearly, we can also assume that each $|X_\alpha|\geq 2$. By connectedness, it follows that each $\dime(X_\alpha)\geq 1$. Therefore, after taking suitable finite subproducts, the requirements of Theorem \ref{theorem_gpp_manifolds} will be satisfied. So we can assume without loss of generality that $X$ has the $\GPP$. As in the proof of Lemma \ref{lemma_gpp}, it is easy to obtain $\Omega_\alpha\in [\kappa]^\omega$ for $\alpha\in\kappa$ such that the following conditions hold:
\begin{itemize}
\item $\Omega_\alpha\cap\Omega_\beta=\varnothing$ whenever $\alpha\neq\beta$,
\item $\bigcup_{\alpha\in\kappa}\Omega_\alpha=\kappa$,
\item Each $X_{\Omega_\alpha}$ satisfies condition $(\partial)$.
\end{itemize}
Noitce that each $X_{\Omega_\alpha}$ is strongly locally homogeneous by \cite[Theorem 3 and Remark 2]{yang}, and strongly $n$-homogeneous for every $n\in\omega$ by \cite[Theorem 2 and Remark 2]{yang}.\footnote{\,The observation made in Footnote 1 also holds for these properties.} Finally, using the fact that $X$ has the $\GPP$, one sees that $\prod_{\alpha\in\kappa}X_{\Omega_\alpha}$ will have the $\GPP$ as well. To conclude the proof, apply Theorem \ref{theorem_main}.
\end{proof}

\begin{corollary}[Stepr\={a}ns, Zhou]\label{corollary_main_reals}
Let $\kappa<\pppp$ be a cardinal. Then $\RRR^\kappa$ is $\CDH$.
\end{corollary}

\begin{corollary}[Stepr\={a}ns, Zhou]\label{corollary_main_hilbert}
Let $\kappa<\pppp$ be an infinite cardinal. Then $[0,1]^\kappa$ is $\CDH$.
\end{corollary}

\section{Optimality}\label{section_optimality}

In this section, we will show that our main result and its corollaries are sharp. We begin by showing that the cardinal $\pppp$ in the statement of Theorem \ref{theorem_main} is as big as possible. The following result was first proved by Hru\v{s}\'{a}k and Zamora Avil\'{e}s \cite[Theorem 3.3]{hrusak_zamora_aviles} in the case where $\kappa=\pppp$ and each $X_\alpha=2$, and the argument is essentially the same in the general case.

Throughout this section, we will freely apply the Hewitt-Marczewski-Pondiczery theorem mentioned in the introduction. We will also need some preliminary definitions. A subset $D$ of a space $X$ is \emph{sequentially dense} if for every $x\in X$ there exists a sequence of elements of $D$ that converges to $x$. A space $X$ is \emph{strongly sequentially separable} if every countable dense subset of $X$ is sequentially dense in $X$. A space $X$ is \emph{sequentially crowded} if $X$ is non-empty and for every $x\in X$ there exists a sequence of elements of $X\setminus\{x\}$ that converges to $x$.

\begin{theorem}\label{theorem_optimality_p}
Let $\kappa$ be a cardinal such that $\pppp\leq\kappa\leq\cccc$, and let $X_\alpha$ for $\alpha\in\kappa$ be separable metrizable spaces such that each $|X_\alpha|\geq 2$. Then $\prod_{\alpha\in\kappa}X_\alpha$ is separable but not $\CDH$.
\end{theorem}
\begin{proof}
Set $X=\prod_{\alpha\in\kappa}X_\alpha$. By \cite[Theorem 11]{gartside_lo_marsh},\footnote{\,As the remark that precedes it seems to confirm, we believe that ``separable triable spaces'' should be substituted by ``separable metrizable spaces'' in the statement of this result.} we can fix a countable dense subspace $D$ of $X$ that is not sequentially dense in $X$. This means that there exists $x\in X$ that is not the limit of any sequence in $D$. In particular, the space $D\cup\{x\}$ is not sequentially crowded. On the other hand, the space $X$ is sequentially crowded because each $|X_\alpha|\geq 2$ and $\kappa$ is infinite (see \cite[Lemma 4.3]{dobrowolski_krupski_marciszewski}). Therefore, in a recursion of length $\omega$, it is possible to construct a countable dense subspace $E$ of $X$ that is sequentially crowded. Since $D$ and $E$ are obviously non-homeomorphic, the proof is concluded.
\end{proof}

Next, we will show that the assumption regarding compactness in Corollary \ref{corollary_main_zero_dimensional} cannot be dropped.

\begin{proposition}\label{proposition_optimality_compactness}
There exists a zero-dimensional Polish space $X$ such that $X\times 2^\omega$ is not $\CDH$.
\end{proposition}
\begin{proof}
Fix a countable dense subset $Q$ of $\omega^\omega$, and let $Q=\{q_n:n\in\omega\}$ be an enumeration. Define
$$
X=\big\{(n,q_n):n\in\omega\big\}\cup\big(\{\omega\}\times\omega^\omega\big)
$$
with the subspace topology inherited from $(\omega+1)\times\omega^\omega$, then set $Z=X\times 2^\omega$. One can easily verify that $X$ is a $\mathsf{G}_\delta$ subspace of $(\omega+1)\times\omega^\omega$, hence it is a Polish space. Set $U=\{(n,q_n):n\in\omega\}$, and observe that $U$ is a dense subset of $X$ consisting of isolated points. Furthermore, the set $U\times 2^\omega$ consists precisely of the points of $Z$ that have a compact neighborhood in $Z$. Finally, pick countable dense subsets $D$ and $E$ of $Z$ such that $D\subseteq U\times 2^\omega$ and $E\nsubseteq U\times 2^\omega$. It is clear that $(D,E)$ witnesses that $Z$ is not $\CDH$.
\end{proof}

Next, we will show that the assumption of connectedness in Theorem \ref{theorem_yang} and Corollary \ref{corollary_main_manifolds} cannot be dropped. As we mentioned in the introduction, a counterexample was already given in \cite[Example 1]{yang}. However, it seems worthwhile to include the following result, since it provides a much wider class of counterexamples with a simpler proof.

\begin{proposition}\label{proposition_optimality_connectedness}
Let $\kappa\leq\cccc$ be an infinite cardinal, and let $X_\alpha$ and $Y_\alpha$ be separable connected spaces such that each $|X_\alpha|\geq 2$ and $|Y_\alpha|\geq 2$. Then $\prod_{\alpha\in\kappa}(X_\alpha\oplus Y_\alpha)$ is separable but not $\CDH$.
\end{proposition}
\begin{proof}
Set $X_{\alpha,0}=X_\alpha$ and $X_{\alpha,1}=Y_\alpha$ for $\alpha\in\kappa$, and set $Z=\prod_{\alpha\in\kappa}(X_\alpha\oplus Y_\alpha)$. Given $f\in 2^\kappa$, define
$$
C_f=\prod_{\alpha\in\kappa}X_{\alpha,f(\alpha)},
$$
and observe that the sets $C_f$ are precisely the connected components of $Z$. Now fix a countable dense subset $D$ of $Z$. Using the assumption that each $|X_\alpha|\geq 2$ and each $|Y_\alpha|\geq 2$, we can assume without loss of generality that $D\cap C_f$ is infinite for every $f\in 2^\kappa$ such that $D\cap C_f\neq\varnothing$. Now pick $z\in C_f$ for some $f\in 2^\kappa$ such that $D\cap C_f=\varnothing$, then set $E=D\cup\{z\}$. It is easy to realize that $(D,E)$ witnesses that $Z$ is not $\CDH$.
\end{proof}

Next, we will show that the assumption of condition $(\partial)$ in Theorem \ref{theorem_yang} and Corollary \ref{corollary_main_manifolds} cannot be dropped.

\begin{theorem}\label{theorem_optimality_boundaries}
Let $\kappa\leq\cccc$ be a cardinal, and let $X_\alpha$ for $\alpha\in\kappa$ be connected manifolds with boundary. Assume that condition $(\partial)$ fails. Then $\prod_{\alpha\in\kappa}X_\alpha$ is separable but not $\CDH$.
\end{theorem}
\begin{proof}
Set $X=\prod_{\alpha\in\kappa}X_\alpha$. Assume, in order to get a contradiction, that $X$ is~$\CDH$. Since $X$ is also connected, it follows from \cite[Corollary 3.2]{van_mill_article_2013} that $X$ is homogeneous. This contradicts \cite{ganea}.
\end{proof}

The following result gives us one more tool to obtain the $\GPP$. Although we will only use it for the purpose of exhibiting a counterexample (see below), it seems to be of independent interest.

\begin{theorem}\label{theorem_gpp_compact}
Let $\kappa\leq\omega$ be a cardinal, and let $X_\alpha$ be crowded compact metrizable spaces for $\alpha\in\kappa$. Assume that $X_\alpha\times X_\beta$ is strongly locally homogeneous for every $\{\alpha,\beta\}\in [\kappa]^2$. Then $\prod_{\alpha\in\kappa}X_\alpha$ has the $\GPP$.
\end{theorem}
\begin{proof}
Clearly, we can assume that $\kappa\geq 2$. Set $X=\prod_{\alpha\in\kappa}X_\alpha$, and let $D\subseteq X$ be countable. Throughout this proof, we will endow $\HH(X)$ with the subspace topology inherited from $\CC(X,X)$ with the compact-open topology. Given $p=\{d,e\}\in [D]^2$ and $\alpha\in\kappa$, define
$$
U_{p,\alpha}=\{h\in\HH(X):h(d)(\alpha)\neq h(e)(\alpha)\}.
$$
Since $\HH(X)$ is Baire by \cite[Corollary 1.3.13]{van_mill_book} (in fact, it is Polish), it will be enough to prove the following two claims.

\noindent\textbf{Claim 1.} Each $U_{p,\alpha}$ is open in $\HH(X)$.

\noindent\textit{Proof.} Fix $p=\{d,e\}\in [D]^2$ and $\alpha\in\kappa$. Pick $h\in U_{p,\alpha}$. Let $U$ and $V$ be disjoint open subsets of $X_\alpha$ such that $h(d)(\alpha)\in U$ and $h(e)(\alpha)\in V$. The observation that
$$
h\in\big[\{d\},\pi^{-1}_\alpha[U]\big]\cap\big[\{e\},\pi^{-1}_\alpha[V]\big]\cap\HH(X)\subseteq U_{p,\alpha}
$$
clearly concludes the proof. $\blacksquare$

\noindent\textbf{Claim 2.} Each $U_{p,\alpha}$ is dense in $\HH(X)$.

\noindent\textit{Proof.} Fix admissible metrics $\di_\alpha$ on $X_\alpha$ for $\alpha\in\kappa$. Given $\Omega\subseteq\kappa$, denote by $\di^\ast_\Omega$ the admissible metric on $X_\Omega$ obtained by setting
$$
\di^\ast_\Omega(x,y)=\sum_{\alpha\in\Omega}2^{-\alpha}\cdot\di_\alpha\big(x(\alpha),y(\alpha)\big)
$$
for $x,y\in X_\Omega$. Fix $p=\{d,e\}\in [D]^2$ and $\alpha\in\kappa$. Pick $h\in\HH(X)$. We need to find elements of $U_{p,\alpha}$ that are arbitrarily close to $h$. Clearly, we can assume that $h(d)(\alpha)=h(e)(\alpha)$. Since $h(d)\neq h(e)$, we can fix $\beta\in\kappa$ such that $h(d)(\beta)\neq h(e)(\beta)$. Set $\Omega=\{\alpha,\beta\}$. Given $x\in X_\Omega$ and $\varepsilon>0$, we will denote by $\Ball(x,\varepsilon)=\{z\in X_\Omega:\di^\ast_\Omega(z,x)<\varepsilon\}$ the open ball around $x$ of radius $\varepsilon$.

Fix $\varepsilon>0$. We will find $h'\in\HH(X)$ such that $h'\circ h\in U_{p,\alpha}$ and $\di^\ast_\kappa\big(h'(x),x\big)<\varepsilon$ for every $x\in X$. By Theorem \ref{theorem_compact_open}, this will conclude the proof. Set $x=\tau_\Omega\big(h(d)\big)$, $x'=\tau_\Omega\big(h(e)\big)$, and $x_\alpha=h(d)(\alpha)=h(e)(\alpha)$. Since $X_\Omega$ is strongly locally homogeneous, there exists $0<\delta<\varepsilon/2$ such that for every $y\in\Ball(x,\delta)$ there exists $h_\Omega\in\HH(X_\Omega)$ such that $h_\Omega(x)=y$ and $\supp(h_\Omega)\subseteq\Ball(x,\delta)$. Without loss of generality, assume that $x'\notin\Ball(x,\delta)$. Since $X_\alpha$ is crowded, we can fix $y\in\Ball(x,\delta)$ such that $y(\alpha)\neq x_\alpha$, then obtain $h_\Omega$ as above. Finally, let $h_\gamma$ be the identity on $X_\gamma$ for $\gamma\in\kappa\setminus\Omega$, then set
$$
h'=h_\Omega\times\prod_{\gamma\in\kappa\setminus\Omega}h_\gamma.
$$
It is straightforward to check that $h'$ is as desired. $\blacksquare$
\end{proof}

In the following table, we provide explicit counterexamples showing that none of the requirements in the statement of Theorem \ref{theorem_main} can be dropped. The first row of the table simply lists these requirements; in the second row, below each requirement, we exhibit a product that satisfies all the \emph{other} requirements but is not $\CDH$. (In the case of $\QQQ$ and \cite[Example 1.3]{van_mill_article_2011}, the product consists of a single factor.)

\begin{center}
\begin{tabular}{c|c|c|c|c}
& & Strongly & Strongly & General \\
$\kappa<\pppp$ & Polish & locally & $n$-homogeneous & position \\
& & homogeneous & for every $n\in\omega$ & property \\\hline
& & & & \\
$(2^\omega)^\pppp$ & $\QQQ$ & \cite[Example 1.3]{van_mill_article_2011} & $(\SSS^1\oplus\SSS^1)^\omega$ & $2^\omega\times\SSS^2$\\
& & &
\end{tabular}
\end{center}

We conclude by clarifying why the spaces in question have the desired properties, in the cases where such clarifications are needed. For the zero-dimensional spaces, the homogeneity properties follow from Proposition \ref{proposition_zero_dimensional}. The $\GPP$ of $(2^\omega)^\pppp$ follows from Theorem \ref{theorem_gpp_groups}, while the $\GPP$ of $(\SSS^1\oplus\SSS^1)^\omega$ follows from Theorem \ref{theorem_gpp_compact}. The fact that $\SSS^2$ is strongly $n$-homogeneous for every $n\in\omega$ can be deduced from \cite[Theorem 3.2]{ungar}, since removing finitely many points from $\SSS^2$ yields a connected space by \cite[Corollary 3.7.3]{van_mill_book}. With the exception of \cite[Example 1.3]{van_mill_article_2011}, all products are non-$\CDH$ by the results of this section or even simpler arguments.

\end{document}